%% file: main_revised.tex
\documentclass[11pt]{amsart}
\usepackage[left=1.5in, right=1.5in]{geometry} 
\usepackage[utf8]{inputenc}
\include{canonicalheader}
\newcommand{\rhon}{{\rho_{\vec{n}}}}

\title[Selmer groups for motives at non-ordinary primes]{On the signed Selmer groups for motives at non-ordinary primes {in $\Z_{\MakeLowercase{p}}^2$-extensions}}

\author[J. Ray]{Jishnu Ray}
\address[Ray]{Harish Chandra Research Institute, A CI of Homi Bhabha National Institute,  Chhatnag Road, Jhunsi, Prayagraj (Allahabad) 211 019 India}
\email{jishnuray@hri.res.in}

\author[F. Sprung]{Florian Sprung}
\address[Sprung]{School of Mathematical and Statistical Sciences, Arizona State University Tempe, AZ 85287-1804, USA}
\email{florian.sprung@asu.edu}

\keywords{non-ordinary motives, signed Selmer groups and Coleman maps, control theorem, $\Z_p^2$ extension of an imaginary quadratic field, $\mu$-invariants under congruences.}
\subjclass[2020]{Primary: 11R23,  Secondary: 11G10, 11R20 }
\begin{document}

\maketitle
\begin{abstract}
Generalizing the work of Kobayashi and the second author for elliptic curves with supersingular reduction at the prime $p$,  Büyükboduk and Lei constructed multi-signed Selmer groups over the cyclotomic $\Z_p$-extension of a number field $F$ for more general non-ordinary motives. In particular, their construction applies to abelian varieties over $F$ with good supersingular  reduction  at all the primes of $F$ above $p$. In this article, we scrutinize the case in which $F$ is imaginary quadratic, and prove a control theorem (that generalizes Kim's control theorem for elliptic curves) of multi-signed Selmer groups of non-ordinary motives  over the maximal abelian pro-$p$ extension of $F$ that is unramified outside $p$, which is the $\Z_p^2$-extension of $F$. We apply it to derive a sufficient condition when these multi-signed Selmer groups are cotorsion over the corresponding two-variable Iwasawa algebra. Furthermore, we compare the Iwasawa $\mu$-invariants of multi-signed Selmer groups over the $\Z_p^2$-extension for two such representations which are congruent modulo $p$.
\end{abstract}
\section{Introduction}
The goals of this article are to 
\begin{enumerate}
    \item establish cotorsion of certain Selmer groups associated to motives in the non-ordinary case via a control theorem, and to
    \item prove a vanishing statement of $\mu$-invariants associated to motives that are congruent in an appropriate sense.
\end{enumerate}

\subsection{The objects involved} Suppose $p$ is an odd prime, $F$ be an imaginary quadratic field in which $p$ splits into two primes $\gp$ and $\gp^c$. Let $F_\cyc$ and $F_\infty$ be the cyclotomic $\Z_p$-extension and the unique $\Z_p^2$-extension of $F$; thus $F_\infty$ contains $F_\cyc$. Let $\Omega=\Gal(F_\infty/F)$ and $\Gamma=\Gal(F_\cyc/F)$.  For $E$ an elliptic curve over $\Q$ with supersingular reduction at $p$ and $a_p=0$,  Kobayashi \cite{kobayashi03}  constructed signed Selmer groups $\Sel^\pm(E/\Q_\cyc)$ which are  cotorsion over the Iwasawa algebra $\Z_p[[\Gal(\Q_\cyc/\Q)]]$. For a construction that includes the general case, see \cite{sprung09}, where the modified Selmer groups are called $\sharp,\flat$-signed Selmer groups over $\Q_\cyc$. Generalizing the work of Kobayashi, Kim constructed multi-signed Selmer groups (assuming $a_p=0$) $\Sel^{\pm,\pm}(E/F_\infty)$ over the $\Z_p^2$-extension $F_\infty$ which are also conjectured to be cotorsion as  $\Z_p[[\Omega]]$-modules \cite{kimdoublysigned}. Out of these four signed Selmer groups, only the ones with the same signs $\Sel^{+,+}(E/F_\infty)$ $\Sel^{-,-}(E/F_\infty)$ are known to be cotorsion (see \cite[Remark 8.5]{LeiPal19}). In the general case \cite{sprung16} (i.e. removing the $a_p=0$ condition), the second author has given a construction which shows that at least one of the chromatic Selmer groups is cotorsion.
The local conditions defining these Selmer groups were then later reinterpreted in terms of Coleman maps using the theory of Wach-modules which enables one to define signed Selmer groups over $F_\cyc$ also for modular forms which are  non-ordinary at $p$ (see \cite{leiloefflerzerbes10}, \cite{leiloefflerzerbes11}). 
Generalizing the work of Kobayashi, Büyükboduk and Lei constructed signed Selmer groups over the cyclotomic extension of any number field for motives $\mathcal{M}$ over $\Q$ which are non-ordinary at $p$. Let $\mathcal{M}_p$ be the $p$-adic realization of the motive $\mathcal{M}$ with a $\Z_p$-lattice $T$ which is $G_F$-stable. Set $M^*=T^*\otimes \Q_p/\Z_p$ where $T^*=\Hom(T,\Z_p(1))$ is the Tate dual of $T$. Upon fixing a Hodge-compatible basis of the Dieudonné module attached to $T$ (see Section \ref{sub:HC}), for each $\underline{I}$ varying over a certain set (see \cite[Definition 3.1]{BL17}), Büyükboduk and Lei constructed the signed Selmer group $\Sel_{\underline{I}}(M^*/F_\cyc)$ which they conjecture to be $\Z_p[[\Gamma]]$-cotorsion. When $T$ is the Tate module of an elliptic curve, their construction  recovers Kobayashi's signed Selmer groups \cite[p. 397]{BL17}.
The assumption that the signed Selmer group $\Sel_{\underline{I}}(M^*/F_\cyc)$ is cotorsion as a $\Z_p[[\Gamma]]$-module can be found in many occasions throughout literature, see e.g. results in \cite{ponsinet}.  

\subsection{The control theorem} Suppose that the motive $\mathcal{M}$ is non-ordinary at both primes $\gp$ and $\gp^c$. Using the big dual exponential map of Loeffler--Zerbes for $F_\infty$ constructed in \cite{LZ14}, one can construct the multi-signed Selmer group $\Sel_{\underline{I}}(M^*/F_\infty)$ for the non-ordinary motive $\mathcal{M}$ over the $\Z_p^2$-extension $F_\infty$. We are mainly interested in finding a sufficient condition for this multi-signed Selmer group to be cotorsion as a $\Z_p[[\Omega]]$-module that works in this general setup. Note that in \cite{Dion2022}, the author \textit{assumed} that this Selmer group is cotorsion while proving a two variable algebraic functional equation for this Selmer group (see \cite[Theorem A]{Dion2022}).

 Our main result in this article is to show the following.
 \begin{corollary}
 [see Corollary \ref{prop:check}]
If
the Bloch-Kato Selmer group $\Sel_{\mathrm{BK}}(M^*/F)$ is finite then the Selmer group  $\Sel_{\underline{I}}(M^*/F_\infty)$ is $\Z_p[[\Omega]]$-cotorsion.
 \end{corollary}

The key  ingredient that goes into the proof of the above  corollary is a control theorem (similar to Kim's for elliptic curves \cite{kimdoublysigned}).
We show that 
\begin{theorem} [see Theorem \ref{thm:mainKIM}]
For all but finitely many $s\in \Z$, the kernel and the cokernel of the restriction map $$ \rho_{\vec{n}}:\Sel_{\underline{I}^c}(M_s/F_{\vec{n}}) \rightarrow \Sel_{\underline{I}^c}(M_s/F_\infty)^{\Omega_{\vec{n}}}$$ are bounded as $\vec{n}$ varies.
\end{theorem}

Here, we have fixed
 generators $(\gamma_1,  \gamma_2)$ of $\Omega\cong \Zp^2$ and for a choice of non-negative integers $\vec{n}:=(n_1, n_2),  n_i\geq 0$, we have denoted by $\Omega_{\vec{n}}$ the subgroup of $\Omega$ generated by $(\gamma_1^{n_1}, \gamma_2^{n_2})$, and the extension  $F_{\vec{n}}$ is the extension $F_\infty^{\Omega_{\vec{n}}}$. The Selmer group  $\Sel_{\underline{I}^c}(M_s/F_*)$ is a twisted Selmer group (twisting $M:=T\otimes \Q_p/\Z_p$ by the $s$-th power of the cyclotomic character) defined in Section \ref{sec:twisted} for a certain indexing set $\underline{I}^c$, cf. Remark \ref{rem:mithu}.

\subsection{Congruence and previous works} In the last part of this article, we study congruences of the signed Selmer group $\Sel_{\underline{I}}(M^*/F_\infty)$. For elliptic curves $E$ and $E^\prime$ over $\Q$ with good ordinary reduction at $p$ such that $E[p] \cong E^\prime[p]$ as Galois modules, Greenberg and Vatsal \cite{GreenbergVatsal} initiated the study of such congruences in Iwasawa theory. They showed that the $\mu$-invariant of the Pontryagin dual of $\Sel_p(E/\Q_\cyc)$ vanishes if and only if the $\mu$-invariant of the Pontryagin dual of $\Sel_p(E^\prime/\Q_\cyc)$ vanishes. Kim generalized this result for supersingular elliptic curves for signed Selmer groups  over the cyclotomic $\Z_p$-extension of $\Q$ (see \cite{Kim09}). It was further generalized by Ponsinet to the case of supersingular abelian varieties and the multi-signed Selmer groups of Bübükboduk--Lei \cite{BL17} over the cyclotomic extension of a number field (\cite[Theorem 3.13]{ponsinet}).
 Under additional hypotheses, in \cite[Theorem 4.15]{FilippoSujatha4}  the result of \cite{Kim09} was generalized to a wider context for supersingular elliptic curves, while still working  over the cyclotomic extension, which was further generalized  for supersingular elliptic curves over the $\Z_p^2$-extension of an imaginary quadratic field $F$ where $p$ splits under the assumption of Conjecture A (see \cite[Corollary 4.6]{Hamidi}). 
 Conjecture A is a conjecture regarding the vanishing of $\mu$-invariant of the dual \textit{fine} $p^\infty$-Selmer group over the cyclotomic $\Z_p$-extension of a number field (see \cite{CoatesSujatha_fineSelmer}).

\subsection{Context and contrast of our result on congruence} In this article we generalize both the results \cite[Theorem 3.13]{ponsinet} and \cite[Corollary 4.6]{Hamidi} mentioned above. In particular, we work in the setting of supersingular abelian varieties and multi-signed Selmer groups of Bübükboduk--Lei  over the $\Z_p^2$-extension $F_\infty$. We note that in contrast to Hamidi's result, our theorem stated below is not dependent on Conjecture A.
 
 Hamidi's result \textit{(loc.cit)} were stated for supersingular elliptic curves with $a_{\gp}=a_{\gp^c}=0$ where $\gp$ and $\gp^c$ are two primes of $F$ above $p$ (see \cite[Hyp 1 (iv)]{Hamidi})\footnote{The analogue of the assumption $a_p=0$  in the context of supersingular abelian varieties is an assumption on the shape of the matrix of Frobenius (on a basis of the associated Dieudonné module). This shape is a specific block anti-diagonal form (see \cite[Section 3.2.1]{LP20}) -- the abelian varieties of the $GL(2)$-type provide such examples (see \cite[Section 3.3]{LP20}).}. We remove these assumptions to arrive at the following result.  Let $\mathcal{X}_{\underline{I}}(M^*/F_\infty)$ and $\mathcal{X}_{\underline{I}}({M^\prime}^*/F_\infty)$ be the Pontryagin duals of the Selmer groups $\Sel_{\underline{I}}(M^*/F_\infty)$ and $\Sel_{\underline{I}}({M^\prime}^*/F_\infty)$, where $M'$ comes from a motive $\mathcal{M'}$ with the same residual representation as $\mathcal{M}$. 
\begin{theorem}[see Theorem \ref{them:congruence_mu}]
 {Let $\mathcal{M}$ and $\mathcal{M'}$ have $G_F$-stable lattices $T$ and $T'$ with congruent mod $p$ reduction that satisfy certain mostly  $p$-adic Hodge theoretic hypotheses (H.-T.), (Cryst.), (Tors.), (Fil), and (Slopes) described in detail in the Preliminaries section. } Suppose that $\mathcal{X}_{\underline{I}}(M^*/F_\infty)$ and $\mathcal{X}_{\underline{I}}({M^\prime}^*/F_\infty)$
have the same corank as  $\Z_p[[\Omega]]$-modules. Then the $\mu$-invariant of $\mathcal{X}_{\underline{I}}(M^*/F_\infty)$ vanishes if and only if the $\mu$-invariant of 
$\mathcal{X}_{\underline{I}}({M^\prime}^*/F_\infty)$ vanishes.
\end{theorem}

\begin{corollary}[see Corollary \ref{torsioncorollary}]
 Suppose that the above Selmer groups
are both $\Z_p[[\Omega]]$-cotorsion. Then their $\mu$-invariants either both vanish, or they both don't.
\end{corollary}

Note that the theorem above can be reinterpreted in terms of an Euler characteristic, since the $\mu$-invariant of $\mathcal{X}_{\underline{I}}(M^*/F_\infty)$ is the $p$-adic valuation of the Euler characteristic $\chi\big(\Omega, \mathcal{X}_{\underline{I}}(M^*/F_\infty)(p)\big)$ (see \cite[Corollary 1.7]{Howson_02}). 

Our original goal was to relate the $\Omega$-Euler characteristic $\chi\big(\Omega, \mathcal{X}_{\underline{I}}(M^*/F_\infty)\big)$ with the $\Gamma$-Euler characteristic $\chi\big(\Gamma, \mathcal{X}_{\underline{I}}(M^*/F_\cyc)\big)$ generalizing \cite[Theorem 5.15]{LeiSujatha}. It would be interesting if one can do this.

\subsection{Outline}
In Section \ref{sec:pre}, we  recall the preliminaries including the notion of Yager module which will be an important tool in developing the theory of the two-variable signed Coleman map, the construction of which is given in Section \ref{sec:multiselmer}. Using the signed Coleman map giving the local conditions at the primes of $F$ above $p$, the signed Selmer group $\Sel_{\underline{I}}(M^*/F_\infty)$ is constructed in Section  \ref{sec:multiselmer}. Ranks of certain Iwasawa modules, including the ranks of the images and the kernels of the signed two variable Coleman map, are analyzed in Section \ref{sec:ranks}. In Section \ref{sec:resultmain}, we prove  
 a control theorem (Theorem \ref{thm:mainKIM}) and use it to provide a sufficient condition for the signed Selmer group over the $\Z_p^2$-extension $F_\infty$ to be cotorsion (see Corollary \ref{prop:check}). Finally, in Section \ref{sec: muinvariance} we prove our theorem regarding the $\mu$-invariants under congruence (see Theorem \ref{them:congruence_mu}).

\section*{Acknowledgments.} We thank Cédric Dion, Parham Hamidi, Antonio Lei, Filippo Nuccio and  Gautier Ponsinet for carefully answering various questions.   The first author is supported by the Inspire research grant, the second author by an NSF grant (2001280).

\section{Preliminaries}\label{sec:pre}
The goal of this section is to fix some notations.
\subsection{The setup}\label{sec: setup}
Let $p$ be an odd prime and $F$ be a number field unramified at $p$. Let $\mathcal{M}$ be a motive over $F$ with coefficients in $\Q$ (see \cite{FPR94}). Let $\mathcal{M}_p$ be its $p$-adic realization and fix a $G_F$- stable $\Z_p$-lattice $T$ inside $\mathcal{M}_p$. Let $g=\dim_{\Q_p}(\mathrm{Ind}_F^{\Q}\mathcal{M}_p)$ and $g_+=\dim_{\Q_p}(\mathrm{Ind}_F^{\Q}\mathcal{M}_p)^+$ the dimension of the $+1$ eigenspace under the action of a fixed complex conjugation on $\mathrm{Ind}_F^{\Q}\mathcal{M}_p$. We write $g_{-}=g-g_+$ and $g_v=\dim_{\Q_p}(\mathrm{Ind}_{F_v}^{\Q_p}\mathcal{M}_p)$ where $v$ is a prime of $F$ above $p$. We then have $g=\sum_{v \mid p}g_v$.

\begin{remark}
Suppose $T$ is of rank $d$. Then $g_v=\rank(T)[F_v:\Q_p].$ Then $g=\sum_{v |p}g_v=\sum_{v |p}\rank(T)[F_v:\Q_p]$. If needed, one can also assume $g_+=g_-=\frac{g}{2}$ wherever necessary (see hypothesis (H.P) of \cite{LeiPonsinet2017}).
\end{remark}
Let $T^*=\Hom(T,\Z_p(1))$ be the Tate dual of $T$. We set $M=T \otimes \Q_p/\Z_p$ and $M^*=T^* \otimes \Q_p/\Z_p$. 

For every prime $v$ of $F$ above $p$, we with the hypotheses on $\mathcal{M}$ as in \cite{ponsinet}:
\\

\noindent \textbf{(H.-T.)} The Hodge--Tate weights of $\mathcal{M}_p$, as a $G_{F_v}$-representation, are in  $[0,1]$. 

\noindent \textbf{(Cryst.)} The $G_{F_v}$-representation $\mathcal{M}_p$ is crystalline.

\noindent \textbf{(Tors.)} The cohomology groups $H^0(F_v, T/pT)$ and $H^2(F_v, T/pT)$ are trivial.

We note that the hypotheses \textbf{(Cryst.)} and \textbf{(H.-T.)} are then also satisfied for $\mathcal{M}^*$, which is  the dual of $\mathcal{M}$. The hypothesis  \textbf{(Tors.)} is also satisfied by $T^*$, which is a $G_F$-stable $\Z_p$-lattice inside the $p$-adic realization $\mathcal{M}_p^*$ of $\mathcal{M}^*$.

\subsection{Dieudonné modules}
For a prime $v$ of $F$ dividing $p$, let $\Dcrisv(T)$ be the Dieudonné module associated to $T$ (see \cite[Defn. V.1.1]{berger04}). It is a free $\mathcal{O}_{F_v}$- module of rank $\dim_{\Q_p}(\mathcal{M}_p)$ and is equipped with a filtration of $\mathcal{O}_{F_v}$-modules $(\Fil^i\Dcrisv(T))_{i \in \Z}$ and a Frobenius operator $\phi$. The filtration is given by
$$
\Fil^i \mathbb{D}_{\cris,v}(T) = \begin{cases} 0 & \text{if $i \geq 1$,} \\ \mathbb{D}_{\cris,v}(T) & \text{if $i \leq -1$.} \end{cases}
$$
We set $\Dcrisv(\mathcal{M}_p):=\Dcrisv(T) \otimes \Q_p$ and this is a  Fontaine's filtered $\phi$-module associated to $\mathcal{M}_p.$
We also make the following assumptions.
\\

\noindent \textbf{(Fil.)} $\sum_{v \mid p} \dim_{\Q_p} \Fil^0\Dcrisv(T) \otimes \Q_p =g_-.$

\noindent \textbf{(Slopes)} The slopes of the Frobenius operator $\phi$ on the Dieudonné module $\Dcrisv(\mathcal{M}_p)$ lie inside $(-1,0)$. 

\subsubsection{Hodge-compatible basis}\label{sub:HC} We fix once and for all a \textit{Hodge-compatible} basis of $\Dcrisv(T)$ which is, by definition, a $\Z_p$-basis  $\{u_1,...,u_{g_v}\}$ of $\Dcrisv(T)$ such that $\{u_1,...,u_{d_v}\}$ is a basis for $\Fil^0\Dcrisv(T)$ for some $d_v$. The matrix of the Frobenius $\phi$ with respect to this basis is of the form 
$$
C_{\phi,v} = C_v\left[
\begin{array}{c|c}
I_{d_v} & 0 \\
\hline
0 & \frac{1}{p} I_{g_v-d_v}
\end{array}
\right]
$$
for some $C_v \in \mathrm{GL}_{g_v}(\Z_p)$ and where $I_n$ is the identity $n \times n$ matrix (see \cite[Section 2.2]{BL17}).
We also note that the Dieudonné module $\Dcrisv(T^*)$  satisfies the hypotheses \textbf{(Fil.)} and \textbf{(Slopes)}.

\begin{remark}
Suppose that $A$ is an abelian variety defined over $F$ having good supersingular reduction at all the primes of $F$ dividing $p$. Let $T_p(A)$ be the Tate module of $A$ and $V_p(A):=T_p(A) \tensor \Q_p$. Then the $G_F$-representation $V_p(A)$ and its Galois stable lattice $T_p(A)$ satisfy all the hypotheses \textbf{(H.-T.)}, \textbf{(Cryst.)}, \textbf{(Tors.)}, \textbf{(Fil.)} and
\textbf{(Slopes)} (cf. \cite[Example 1.1]{ponsinet}).
\end{remark} 
\subsection{The Yager module}\label{sec Yager}

Let $L/\Q_p$ be a finite unramified extension. Let $L_\infty$ be any unramified $p$-adic Lie extension of $L$ with Galois group $U$.  For $x \in \mathcal{O}_L$, define 

$$
y_{L/\Q_p}(x) := \sum_{\tau \in \Gal(L/\Qp)} \tau (x) \cdot \tau^{-1} \in \mathcal{O}_L [\Gal(L/\Q_p)]. 
$$

Let $S_{L/\Q_p}$ be the sub-$\mathcal{O}_L [\Gal(L/\Q_p)]$-module generated by the image of $y_{L/\Q_p}$ in $\mathcal{O}_L [\Gal(L/\Q_p)]$.  Then there is an isomorphism of $\Lambda_{\mathcal{O}_L}(U)$-modules (cf.  \cite[Section 3.2]{LZ14})
$$
y_{L_\infty/\Q_p}:\varprojlim_{\Q_p \subseteq L \subseteq L_\infty} \mathcal{O}_L \xrightarrow{\cong} S_{L_\infty / \Q_p}:= \varprojlim_{\Q_p \subseteq L \subseteq L_\infty} S_{L/\Q_p},
$$
where the inverse limit is taken with respect to the trace maps on the left and the projection maps $\Gal(L^\prime/\Q_p) \to \Gal(L/\Q_p)$ for $L \subseteq L^\prime$ on the right. By \cite[Proposition 3.2]{LZ14}, $\varprojlim_{\Q_p \subseteq L \subseteq L_\infty} \mathcal{O}_L$ is a free $\Lambda_{\mathcal{O}_L}(U)$-module of rank one, so that the Yager module $S_{L_\infty/\Q_p}$ is also free of rank one over $\Lambda_{\mathcal{O}_L}(U)$.
The Yager module $S_{L_\infty/\Q_p}$ comes equipped with a compact and Hausdorff topology that coincides with the subspace topology from $\Lambda_{\widehat{\mathcal{O}}_{L_\infty}}(U)$, where $\widehat{\mathcal{O}}_{L_\infty}$ is the completion of the ring of integers of $L_\infty$.

\section{The multi-signed Selmer groups over a $\Z_p^2$-tower}\label{sec:multiselmer}

In this section, we recall some known results, mainly the construction of certain Coleman maps and present a conjecture due to Büyükboduk--Lei concerning the cotorsion of a Selmer group constructed from the Coleman map in the case of a cyclotomic $\Z_p$-extension. We generalize their conjecture to the case of a $\Z_p^2$-extension.

Let $F$ be imaginary quadratic, $F_\infty$ be the compositum of all $\Z_p$-extensions of $F$, so that in particular $F_{\cyc}\subset F_\infty$, where $F_{\cyc}$ is the cyclotomic $\Z_p$-extension of $F$. We let $\Omega:= \Gal(F_\infty/F)\cong \Z_p^2$.   Suppose that the prime $p$ splits into two primes $\gp$ and $\gp^c$ of $F$, where $c$ denotes complex conjugation. We will write $\gq$ to denote either of  these two primes. If $\mathfrak{a}$ is an ideal of $\mathcal{O}_F$, let $F(\mathfrak{a})$ be the ray class field of $F$ of conductor $\mathfrak{a}$. If $n \geq 0$ is an integer, we will write $\Omega_n := \Omega^{p^n}$ and $F_n := F_\infty^{\Omega_n} = F(p^{n+1})^{\Gal(F(1)/F)}$. We assume that $F(1)\bigcap F_\infty=F$,  hence   $\Omega \cong G_{\gp}\times G_{\gp^c}$ where $G_\gq$ is the Galois group of the extension $F(\gq^\infty) \bigcap F_\infty / F$. We fix topological generators $\gamma_\gp$ and $\gamma_{\gp^c}$ respectively for the groups $G_\gp$ and $G_{\gp^c}$. Let $\Lambda(\Omega) := \Zp [[\Omega ]]  \cong \Zp [[ \gamma_\gp-1, \gamma_{\gp^c}-1]]$ be the Iwasawa algebra of $\Omega$. {More generally, for a choice of non-negative integers $\vec{n}:=(n_1, n_2); n_i\geq 0$, denote by $\Omega_{\vec{n}}$ the subgroup of $\Omega$ generated by $(\gamma_{\gp}^{n_1}, {\gamma_{\gp^c}}^{n_2})$, and let $F_{\vec{n}}:=F_\infty^{\Omega_{\vec{n}}}$.}

Write $\Omega_\gq$ for the decomposition group of $\gq$ in $\Omega$.  
Let $\Sigma$ be a finite set of primes of $F$ containing the primes dividing $p$, the archimedean primes and the primes of ramification of $M^*$. Let $F_\Sigma$ be the maximal extension of $F$ that is unramified outside $\Sigma$. If $w$ is a place of $F_\infty$ above $\gq$, then by local class field theory, the Galois group $\Gal(F_{\infty,w}/F_\gq)$ is isomorphic to $\Z_p^2$. 
Hence the extension $F_{\infty,w}$ coincides with the compositum of the cyclotomic $\Z_p$-extension of $\Q_p$ and the unramified $\Z_p$-extension of $\Q_p$.

Let $L_\infty$ and $k_\text{cyc}$ be the unramified $\Z_p$-extension and the cyclotomic $\Z_p$-extension of $\Q_p$ respectively. Let $k_\infty$ be the compositum of $L_\infty$ and $k_\cyc$.  For $n \geq 0$, let $k_n$ and $L_n$ be the subextensions of $k_\cyc$ and $L_\infty$ respectively such that $[k_n:\Q_p]=p^n$ and $[L_n:\Q_p]=p^n$. We set $\Omega_p := \Gal(k_\infty/\Q_p) \cong \Z_p^2$, $\Gamma_{\text{ur}} := \Gal(L_\infty/\Q_p) \cong \Z_p$ and $\Gamma_{\text{cyc}} := \Gal(k_\text{cyc}/\Q_p) \cong \Z_p$. Thus $\Omega_p \cong \Gamma_{\mathrm{ur}} \times \Gamma_{\mathrm{cyc}}$.  For $\vec{n}=(n_1,n_2);n_i\geq 0$, on fixing two topological generators $(\gamma_1,\gamma_2)$ of $\Omega_p\cong \Z_p^2$, we will write $\Omega_{p,\vec{n}}$ to denote the subgroup of $\Omega_p$ generated by $(\gamma_1^{n_1},\gamma_2^{n_2})$.  Choosing a topological generator $\gamma_2$ of $\Gamma_\cyc$, we let $\mathcal{H}(\Gamma_\text{cyc})$ be the set of power series
$$
\sum_{n\geq 0} c_{n} \cdot (\gamma_2 -1 )^n
$$
with coefficients in $\Q_p$ such that $\sum_{n\geq 0}c_{n}X^n$ converges on the open unit disk.

Fix a Hodge-compatible basis $\{u_1,...,u_{g_\gq}\}$ of $\Dcrisq(T)$. Let $\{Y_{L_\infty/\Q_p}\}$ be a basis of the Yager module $S_{L_\infty/\Q_p}$. Then Loeffler--Zerbes \cite{LZ14} constructed a two-variable big logarithm map $$
\mathcal{L}_{T,\gq}^{\infty}: H^1_\mathrm{Iw}(k_\infty,T) \to Y_{L_\infty/\Q_p} \cdot \left( \mathcal{H}(\Gamma_\text{cyc})\widehat{\otimes}\Lambda(\Gamma_\text{ur}) \right) \otimes_{\Z_p} \Dcrisq(T).
$$
which interpolates the Bloch--Kato dual exponential map \cite[Theorem 4.15]{LZ14}.

Let $K$ be a finite unramified extension of $\Q_p$ and let $K_\mathrm{cyc}$ denote the cyclotomic $\Z_p$-extension of $K$ with  Galois group $\Gamma_{\mathrm{cyc}}$. In \cite{BL17}, the authors show the existence of one-variable Coleman maps
$$
\col_{T,K,i}:H^1_\mathrm{Iw}(K_\mathrm{cyc},T) \to \mathcal{O}_K\otimes \Lambda(\Gamma_{\mathrm{cyc}})
$$
for $1 \leq i \leq g_v$. Those maps are compatible with the corestriction maps $$H^1_\mathrm{Iw}(L_{m,\mathrm{cyc}},T)\to H^1_\mathrm{Iw}(L_{m-1,\mathrm{cyc}},T)$$ and the trace maps $\mathcal{O}_{L_m}\otimes \Lambda(\Gamma_\mathrm{cyc}) \to \mathcal{O}_{L_{m-1}}\otimes \Lambda(\Gamma_\mathrm{cyc})$, where $L_{m,\cyc}$ is the cyclotomic $\Z_p$-extension of $L_m$. We define the two-variable Coleman maps by taking the inverse limit of  $\col_{T,L_m,i}$ as $L_m$ runs through the finite extensions between $L_\infty$ and $\Q_p$. In order to get a family of maps landing in $\Lambda(\Gamma_p)$, we further compose it with the Yager module $S_{L_\infty/\Q_p}$. More precisely, the two-variable Coleman maps are defined by
\begin{align}
\col_{T,i}^\infty:H^1_\mathrm{Iw}(k_\infty,T) \cong \varprojlim_{L_m} H^1_{\mathrm{Iw}}(L_{m,cyc},T) &\to Y_{L_\infty/\Q_p}\cdot \Lambda(\Omega_p) \label{two-variable-coleman}\\
(z_m) & \mapsto (y_{L_\infty/\Q_p}\otimes 1) \circ (\varprojlim_{L_m} \col_{T,L_m,i}(z_m))
\notag.
\end{align}

 We can identify $Y_{L_\infty/\Q_p} \cdot \Lambda(\Omega_p)$ with $\Lambda(\Omega_p)$, and hence omit the basis $Y_{L_\infty/\Q_p}$ from the notation and see the Coleman maps $\col_{T,i}^\infty$ as taking value in $\Lambda(\Omega_p)$.
{By combining \cite[Theorem 2.13]{BL17} and \cite[Theorem 4.7 (1)]{LZ14}} and including the prime $\gq$ in our notation,  we have the Coleman maps
\begin{equation}\label{eq:coleman}
\col_{T,\gq,i}^\infty: H^1(K_\gq, T \otimes \Lambda(\Omega_\gq)^\iota) \rightarrow \Z_p[[\Omega_p]]
\end{equation}
for $i\in \{1,...,g_\gq\}$
 such that 
$$
\mathcal{L}_{T,\gq}^{\infty} = (u_1,\ldots, u_{g_\gq}) \cdot M_{T,\gq} \cdot \begin{bmatrix} \col_{T,\gq, 1}^\infty \\ \vdots \\ \col_{T,\gq,g_\gq}^\infty \end{bmatrix}
.$$

Here the matrix $M_{T,\gq}$ is defined in the following way.
For $n \geq 1$, as in \cite{BL17}, we can define 
$$
C_{\gq, n} = \left[
\begin{array}{c|c}
I_{d_\gq} & 0 \\
\hline
0 & \Phi_{p^n}(1+X)I_{g_\gq-d_\gq}
\end{array}
\right] C_\gq^{-1}
$$
and $$M_{\gq,n}=(C_{\phi,\gq})^{n+1} C_{\gq,n}\cdots C_1,$$
where 
$\Phi_{p^n}$ is the $p^n$-th cyclotomic polynomial and the matrices $C_{\phi,\gq}$ and $C_\gq$ are as in Section \ref{sub:HC}. By \cite[Proposition 2.5]{BL17} the sequence $(M_{\gq,n})_{n \geq 1}$ converges to some $g_\gq \times g_\gq$ logarithmic matrix $M_{T,\gq}$ with entries in $\mathcal{H}(\Gamma_{\cyc})$.

\begin{remark}
By assumption, $F_\infty \cap F(1)=F$, which happens if $p$ does not divide $h_F$, the class number of $F$. For  a prime $\gq$ of $F$ over $p$, $\gq$ does not split in $F_\cyc$ and every prime above $p$ in the anticyclotomic extension of $F$ is totally ramified. Hence the assumption  $F_\infty \cap F(1)=F$ implies that there is a \textit{unique} prime above $\gq$ in $F_\infty$. However, for the definition of the multi-signed Selmer groups given below, we do not need the assumption  $F_\infty \cap F(1)=F$.
\end{remark}
Denote by $p^t$ the number of primes above $\gp$ and $\gp^c$ in $F_\infty$.
Fix a choice of coset representatives $\gamma_1,\ldots, \gamma_{p^t}$ and $\delta_1,\ldots,\delta_{p^t}$ for $\Omega/\Omega_\gp$ and $\Omega/\Omega_{\gp^c}$ respectively.
Since $p$ splits in $F$, we can identify $F_\gq$ with $\Q_p$ and $\Omega_\gq$ with $\Omega_p$. Consider the ``semi-local" decomposition coming from Shapiro's lemma
$$
H^1(F_\gp, T\otimes \Lambda(\Omega)^\iota) = \bigoplus_{j=1}^{p^t} H^1(F_\gp, T \otimes \Lambda(\Omega_\gp)^\iota)\cdot \gamma_j \cong \bigoplus_{v | \gp}H^1_\text{Iw}(F_{\infty,v},T),
$$
where $v$ runs through the primes above $\gp$ in $F_\infty,$ and $\iota: \Lambda(\Omega) \rightarrow \Lambda(\Omega)$ is the involution obtained by sending $g \in \Omega$ to $g^{-1}$. By choosing a Hodge-compatible basis of $\mathbb{D}_{\mathrm{cris},\gp}(T)$, define the Coleman map for $T$ at $\gp$ by
\begin{align*}
\col_{T,\gp,i}^{k_\infty}: H^1(F_\gp,T \otimes \Lambda(\Omega)^\iota) &\to \Lambda(\Omega) \\
x=\sum_{j=1}^{p^t}x_j \cdot \gamma_j &\mapsto \sum_{j=1}^{p^t} \col_{T,\gp,i}^\infty(x_j)\cdot \gamma_j
\end{align*}
for all $1 \leq i \leq g_\gp$.

Let $\mathcal{L}_{T,\gp}^{k_\infty}=\oplus_{j=1}^{p^t}\mathcal{L}_{T,\gp}^{\infty}\cdot \gamma_j$. Define $\col_{T,\gp^c,i}^{k_\infty}$ and $\mathcal{L}_{T,\gp^c}^{k_\infty}$ in an analogous way. Let $I_\gq$ denote a subset of $\{1,\ldots,g_\gq\}$ and let 
\begin{align*}
\col_{T,I_\gq}^\infty : H^1(F_\gq,T \otimes \Lambda(\Omega)^\iota) & \to \bigoplus_{i=1}^{|I_\gq|} \Lambda(\Omega) \\
z & \mapsto (\col^{k_\infty}_{T,\gq,i}(z))_{i \in I_\gq}.
\end{align*}
Tate's local pairing induces a pairing
\begin{equation}\label{TatePairing}
\bigoplus_{w|\gq}H^1_\mathrm{Iw}(F_{\infty,w},T) \times \bigoplus_{w|\gq}H^1(F_{\infty,w},M^*) \to \Q_p/\Z_p
\end{equation}

for all places $w$ of $F_\infty$ above $\gq$. We define $H^1_{I_\gq}(F_{\infty,\gq},M^*) \subseteq \bigoplus_{v | \gq}H^1(F_{\infty,v},M^*)$ as the orthogonal complement of $\ker \col_{T,I_\gq}^\infty$ under the previous pairing.

Let $K^\prime$ be a finite extension of $\Q_p$. We define
$$H^1_{\mathrm{un}}(K^\prime, \mathcal{M}_p^*)=\ker\Big(H^1(K^\prime, \mathcal{M}_p^*) \rightarrow H^1(K_{\mathrm{ur}}, \mathcal{M}_p^*) \Big)$$ where $K_\mathrm{ur}$ is the maximal unramified extension of $K^\prime$. Let $H^1_{\mathrm{ur}}(K,M^*)$ be the image of $H^1_{\mathrm{ur}}(K,\mathcal{M}_p^*)$ under the natural map 
$$H^1_{\mathrm{ur}}(K,\mathcal{M}_p^*) \rightarrow H^1(K,M^*).$$
For an infinite extension $L$ of $\Q_p$, we define 
$$H^1_{{\mathrm{ur}}}(L,M^*) = \varinjlim_{K^\prime}H^1_{\mathrm{ur}}(K^\prime, M^*)$$
where $K^\prime$ runs through all the finite subextensions of $L$ and the limit is taken with respect to restriction maps.

Let $\Sigma^\prime$ be the set of primes of $F_\infty$ above $\Sigma$. 
Let $\mathcal{I}_p$ be the set of tuples $\underline{I}=(I_\gp, I_{\gp^c})$ such that $I_\gq $ is a subset of $\{1,...,g_\gq\}$ such that $|I_\gp|+|I_{\gp^c}|=g_{-}$.

\begin{definition}
For any $\underline{I}$ as above, set

$$\mathcal{P}_{\Sigma,\underline{I}}(M^*/F_\infty)=\prod_{w \in \Sigma^\prime, w \nmid p}\frac{H^1(F_{\infty,w}, M^*)}{H^1_{\mathrm{ur}}(F_{\infty,w}, M^*)} \times \prod_{\gq \mid p}\frac{\bigoplus_{w \mid \gq}H^1(F_{\infty,w}, M^*)}{H^1_{I_\gq}(F_{\infty,\gq},M^*)}.$$

Then the multi-signed Selmer group over the $\Z_p^2$-extension $F_\infty$ is defined as 
$$\Sel_{\underline{I}}(M^*/F_\infty)=\ker\Big(H^1(F_\Sigma/F_\infty, M^*) \rightarrow \mathcal{P}_{\Sigma,\underline{I}}(M^*/F_\infty)\Big).$$
\end{definition}

Our construction generalizes the construction of the multi-signed Selmer groups for the cyclotomic $\Z_p$-extension $F_\cyc$ given in \cite[Section 1.5 and Section 1.6]{ponsinet} (originally constructed by Büyükboduk--Lei \cite{BL17}). For a prime $\gq$ of $F$ above $p$,  we have the signed Coleman map for the cyclotomic extension 
$$\col_{T,I_\gq}^\cyc: \HIw(F_\gq,T) \rightarrow \bigoplus_{i=1}^{|I_\gq|}\Z_p[[\Gal(F_\cyc/F)]]$$ which give the local condition $H^1_{I_\gq}(F_{\cyc,\gq},M^*)$. This enables the author to define the signed Selmer group $\Sel_{\underline{I}}(M^*/F_\cyc)$ (see \cite[Definition 1.7]{ponsinet}) 
(Ponsinet uses $F_\infty$ to denote the cyclotomic $\Z_p$-extension and uses $\col_{T,I_\gq}$ in place of $\col_{T,I_\gq}^\cyc$; for convenience in this article we have changed the notations of Ponsinet a bit.) 

We can also define the signed Selmer groups at finite layers. We now assume that 
$F_\infty \cap F(1)=F$ and hence, by abuse of notation, we will let $\gq$ denote the unique prime at $F_\infty$ over the prime $\gq\in \{\gp,\gp^c\}$ of $F$.

We define $H^1_{I_\gq}(F_{\vec{n},\gq},M^*):=H^1_{I_\gq}(F_{\infty, \gq},M^*)^{\Omega_{\vec{n}}}$, just as in \cite[Definition 1.5]{ponsinet}.  Defining the local condition at a finite layer by using the local condition at an infinite tower and then descending to the finite level is also crucially exploited in the work of Burungale--Büyükboduk--Lei (see \cite[Remark 4.9]{BBL2022}). This approach also appears in many other works of Lei, for example see \cite[Definition 5.2 and Remark 5.4]{Lei2023}. We will take this approach since it will have several advantages as we will see in the control theorem later (see Theorem \ref{thm:mainKIM}).
Having defined the analogous local condition $\mathcal{P}_{\Sigma,\underline{I}}(M^*/F_{\vec{n}})$ as before using  $H^1_{I_\gq}(F_{\vec{n},\gq},M^*)$, we can define the Selmer group 

$$\Sel_{\underline{I}}(M^*/F_{\vec{n}}):=\ker\Big(H^1(F_\Sigma/F_{\vec{n}}, M^*) \rightarrow \mathcal{P}_{\Sigma,\underline{I}}(M^*/F_{\vec{n}})\Big).$$

Recall that $\vec{n}=(n_1,n_2)$ and if $n_1=n_2=n$, then we will simply write $\vec{n}$ as $n$. 
 Note that, since 
$H^1_{I_\gq}(F_{\infty,\gq},M^*)$ is a discrete $\Omega_p$-module, by \cite[Proposition 1.1.8]{Neukirch}, 
$$H^1_{I_\gq}(F_{\infty,\gq},M^*)=\varinjlim_n H^1_{I_\gq}(F_{\infty,\gq},M^*)^{\Omega_{p,n}}=\varinjlim_n H^1_{I_\gq}(F_{n,\gq},M^*).$$ 
Since taking direct limits is an exact functor, this gives
$$\varinjlim_n \frac{H^1(F_{n,\gq},M^*)}{H^1_{I_\gq}(F_{n,\gq},M^*)}\cong \frac{\varinjlim_n H^1(F_{n,\gq},M^*)}{\varinjlim_n H^1_{I_\gq}(F_{n,\gq},M^*)}\cong \frac{H^1(F_{\infty,\gq},M^*)}{H^1_{I_\gq}(F_{\infty,\gq},M^*)}.$$
Furthermore, $\varinjlim_n H^1_{\mathrm{ur}}(F_{n,v},M^*)=H^1_{\mathrm{ur}}(F_{\infty,v},M^*)$ when $v \nmid p$. Hence we obtain the following compatibility  between multi-signed Selmer groups at finite levels $F_n$ and the infinite $\Z_p^2$-extension $F_\infty$. We have 
$$\Sel_{\underline{I}}(M^*/F_\infty)\cong \varinjlim_n \Sel_{\underline{I}}(M^*/F_n)$$
as $\Lambda[[\Omega]]$-modules.

The following conjecture has been made by Büyükboduk and Lei \cite[Remark 3.57]{BL17}.
\begin{conj}\label{conj1}
 For any $\underline{I} \in \mathcal{I}_p$, the signed Selmer group $\Sel_{\underline{I}}(M^*/F_\cyc)$ is cotorsion over the Iwasawa algebra $\Z_p[[\Gal(F_\cyc/F)]]$.
\end{conj}
\begin{remark}
When the base field $F$ is $\Q$ and $\mathcal{M}$ is the Tate module of a supersingular elliptic curve with $a_p=0$, the signed Selmer groups $\Sel_{\underline{I}}(M^*/F_\cyc)$  for $\underline{I} \in \mathcal{I}_p$ coincides with Kobayashi's plus and minus Selmer groups (see \cite[Appendix 4]{BL17}). Kobayashi's plus and minus Selmer groups are known to be cortorsion \cite{kobayashi03} and hence  conjecture \ref{conj1} holds. Moreover, conjecture \ref{conj1} also holds for at least one of the chromatic Selmer groups in the case of $p$-supersingular elliptic curves with $a_p \neq 0$ and eigenforms which are non-ordinary at $p$. (see  \cite[Proposition 6.14, Theorem 7.14]{sprung09} and \cite[Theorem 6.5]{leiloefflerzerbes10}).
\end{remark}

Generalizing Conjecture \ref{conj1} above in the case of $\Z_p^2$-extension seems reasonable to make. 
\begin{conj}\label{conj:tors}
For any $\underline{I} \in \mathcal{I}_p$, the Selmer group $\Sel_{\underline{I}}(M^*/F_\infty)$ is $\Lambda(\Omega)$-cotorsion.
\end{conj}
\begin{remark}
Suppose that $\mathcal{M}$ is the Tate module of a supersingular elliptic curve. Then the two-variable Coleman maps we defined over the $\Z_p^2$-extension $F_\infty$ are the same as as the $\sharp/\flat$-Coleman maps in \cite[Section 5]{LeiSprung2020}.

For $\bullet,\circ \in \{\sharp,\flat\}$, one can construct  signed $p$-adic $L$-functions $\mathfrak{L}_p^{\bullet,\circ}(E/F)$. Under the hypothesis  that $\mathfrak{L}_p^{\bullet,\circ}(E/F)$ is nonzero,  \cite[Proposition 2.23]{sprung16} or \cite[Theorem 3.7]{castella2018iwasawa} shows that the signed Selmer group $\Sel^{\bullet,\circ}(E/F_\infty)$ is $\Lambda(\Omega)$-cotorsion. The $\sharp/\flat$-Selmer groups generalise Kim's $\pm/\pm$-Selmer groups \cite{kimdoublysigned}, who assumed $a_p=0$. The Selmer groups for $\underline{I}=(1,1)$ and $\underline{I}=(2,2)$ correspond to $+/+$ and $-/-$ Selmer groups and these are known to be cotorsion (see \cite[Remark 8.5]{LeiPal19}) giving partial results toward conjecture \ref{conj:tors}.
\end{remark}
\begin{remark}\label{rem:mithu}
Recall that $\mathcal{I}_p$ was the set of tuples $\underline{I}=(I_\gp, I_{\gp^c})$ such that $I_\gq $ is a subset of $\{1,...,g_\gq\}$ for $\gq \in \{\gp, \gp^c\}$ satisfying $|I_\gp|+|I_{\gp^c}|=g_{-}$. Denote by $I_\gq^c$ the complement of the subset $I_\gq$. Then $\underline{I}^c=(I_\gp^c, I_{\gp^c}^c)$ satisfies $|I_\gp^c|+|I_{\gp^c}^c|=g-g_{-}=g_+$. One can also define Coleman maps $\col_{T^*,I_\gq}^\infty$ and construct the signed Selmer groups $\Sel_{\underline{I}^c}(M/F_\infty)$ for $M$ (for details, see \cite[Lemma 3.14]{Dion2022}).
Also note that Conjecture \ref{conj:tors} is expected to hold for $\Sel_{\underline{I}^c}(M/F_\infty)$. Under certain conditions, for supersingular abelian varities of $GL(2)$-type, there is also a functional equation relating the characteristic ideal of the dual of  $\Sel_{\underline{I}}(M^*/F_\infty)$ with that of $\Sel_{\underline{I}^c}(M/F_\infty)$ (see \cite[Theorem A]{Dion2022}).
\end{remark}

\subsection{Twisted Selmer groups}\label{sec:twisted}
We follow the notations as in \cite[discussion after Remark 1.10]{ponsinet}. For $s \in \Z$, set $M_s^*=M^* \otimes \chi^s$ where $\chi$ is the cyclotomic character. Note that the cyclotomic character factors through the cyclotomic extension of $F$.  We have $M_s^*=M^*$ as a $\Gal(\overline{F}/F_\infty)$-module and hence $H^1(F_\infty,M_s^*)=H^1(F_\infty,M^*) \otimes \chi^s$. 

For every prime $v$ of $F$, we also have $H^1(F_{v,\infty},M_s^*)=H^1(F_{v,\infty},M^*) \otimes \chi^s$ and $H^0(F_{v,\infty},M_s^*)=0$. At any prime $\gq$ of $F$ dividing $p$,  we set 
$$H^1_{I_\gq}(F_{\infty,\gq},M_s^*)=H^1_{I_\gq}(F_{\infty,\gq},M^*) \otimes \chi^s.$$ We can now define multi-signed Selmer groups with these twisted local conditions. As in \cite[p. 1643]{ponsinet}, we also note that $\Sel_{\underline{I}}(M_s^*/F_\infty) \cong \Sel_{\underline{I}}(M^*/F_\infty)\otimes \chi^s$ as $\Lambda[[\Omega]]$-modules. 

\section{Ranks of Iwasawa modules}\label{sec:ranks}
Recall that $\gq \in \{\gp,\gp^c\}$, put $H=\Gal(F_\infty/F_\cyc)$, and $\Gamma=\Omega/H=\Gal(F_\cyc/F)$.

\begin{proposition}\label{prop:rankIwasawa}
For $\gq \in \{\gp,\gp^c\}$, we have \begin{enumerate}
    \item $\rank_{\Z_p[[\Gamma]]}H^1_{\mathrm{Iw}}(F_{\cyc,\gq},T)=g_\gq$
    
\item The torsion sub-$\Z_p[[\Gamma]]$-module of $H^1_{\mathrm{Iw}}(F_{\cyc,\gq},T)$ is isomorphic to $T^{G_{F_\cyc}}$.
\item $\rank_{\Z_p[[\Omega_p]]}H^1_{\mathrm{Iw}}(F_{\infty,\gq},T)=g_\gq$
\end{enumerate}
\end{proposition}
\begin{proof}
(1) and (2) is due to Perrin-Riou (see \cite[A.2]{PR95}).

Note that, by \cite[p. 23]{BL21} and the proof of Lemma 2.16 of \cite{BL21}, we have 
$$H^1_{\mathrm{Iw}}(F_{\infty,\gq},T)\cong \big(N_{L_\infty}(T)^{\psi=1}\big)^\Delta \cong \big(N_{\Q_p}(T)^{\psi=1}\big)^\Delta \widehat{\otimes} S_{L_\infty/\Q_p} \cong H^1_{\mathrm{Iw}}(F_{\cyc,\gq},T) \widehat{\otimes} S_{L_\infty/\Q_p},$$
where $\Delta=\Gal(L_\infty(\mu_{p^\infty})/k_\infty)$ is of order $p-1$, 
$N_{L_\infty}(T)$ and $N_{\Q_p}(T)$ are Wach modules as in \textit{(loc.cit)} and $\psi$ (à la Fontaine) is the left inverse of the Frobenius $\phi$-operator on the respective Wach modules. Finally, note that the Yager module is a free module of rank one over $\Lambda_{\Q_p}(U)$ (cf. section \ref{sec Yager}). This proves (3). 
\end{proof}
\begin{lemma}\label{lemmacheck}
The kernel and the  image of $\col_{T,I_\gq}^\infty$ are $\Z_p[[\Omega_p]]$-modules of rank $g_\gq - |I_\gq|$ and $|I_\gq|$ respectively. 
\end{lemma}
\begin{proof}
By the proof of Lemma 2.16 of \cite{BL21}, $$\im(\col_{T,I_\gq}^\infty) \cong \im(\col_{T,I_\gq}^\cyc) \widehat{\otimes} S_{L_\infty/\Q_p.}$$
and $\im(\col_{T,I_\gq}^\cyc) $ is of rank $|I_\gq|$   contained in a free  $\Z_p[[\Gamma_\cyc]]$-module of finite index (see \cite[Lemma 1.2]{ponsinet}).
{Hence $\im(\col_{T,I_\gq}^\infty)$ is of rank $|I_\gq|$  contained in a free $\Z_p[[\Omega_p]]$-module of finite index.}
\end{proof}

\section{A Kim-type control theorem and a mini-control theorem}\label{sec:resultmain}
 
We prove two results, a Kim-type control theorem, generalizing B.D. Kim's theory for elliptic curves, and a corollary, which is a mini-control theorem that shows that the finitenes of a Selmer group at base level implies cotorsion of the same Selmer group in our towers of number fields.
The first theorem below (Theorem \ref{thm:mainKIM})
generalizes a control theorem of Ponsinet which was over the cyclotomic extension (see \cite[Lemma 2.3]{ponsinet}).

\begin{theorem} \label{thm:mainKIM}
For all but finitely many $s\in \Z$, the kernel and the cokernel of the restriction map $$ \rho_{\vec{n}}:\Sel_{\underline{I}^c}(M_s/F_{\vec{n}}) \rightarrow \Sel_{\underline{I}^c}(M_s/F_\infty)^{\Omega_{\vec{n}}}$$ are bounded as $\vec{n}$ varies.
\end{theorem}
\begin{proof}
By definition, we have the following commutative diagram:

	\begin{equation}\label{Selmer_definition}
	\begin{tikzcd}
	0 \arrow[r] & \Sel_{\underline{I}^c}(M_s/F_{\vec{n}}) \arrow[d, "\rho_{\vec{n}}"] \arrow[r] &  H^1(F_\Sigma/F_{\vec{n}}, M_s) \arrow[d, "h"]  \arrow[r, "\lambda"] & \mathcal{P}_{\Sigma,\underline{I}^c}(M_s/F_{\vec{n}}) \arrow[d, "\Xi_{\vec{n}}"] \\
		0 \arrow[r] & \Sel_{\underline{I}^c}(M_s/F_{\infty})^{\Omega_{\vec{n}}} \arrow[r] & H^1(F_\Sigma/F_{\infty}, M_s)^{\Omega_{\vec{n}}}  \arrow[r] &  \mathcal{P}_{\Sigma,\underline{I}^c}(M_s/F_\infty)^{\Omega_{\vec{n}}} 
	\end{tikzcd}
	\end{equation}
 By the assumption \textbf{(Tors)},   the fact that $\Gal(F_{\infty,\gq}/F_\gq)$ is a pro-$p$ group along with the orbit-stabilizer theorem implies that    $H^0(F_{\infty,\gq},M_s)=0$ and hence $\ker(h)=0$ and $\coker(h)=0$. Thus, we have that $\ker(\rhon)=0$ and $\coker(\rhon)=\ker(\Xi_{\vec{n}})\cap\im(\lambda)$ by the snake lemma.  Thus, it suffices to prove the theorem for $\ker(\Xi_{\vec{n}})$, which we now study one prime  $v$ at a time. We show that the $v$-component is almost always $0$, and then analyze the $v$-component when it is not. To this end, we consider the commutative diagram

	\begin{equation}\label{local}
	\begin{tikzcd}
	0 \arrow[r] & H^1_{*}(F_{\vec{n},v},M_s) \arrow[d, "\ell"] \arrow[r] &  H^1(F_{\vec{n},v},M_s) \arrow[d, "g"]  \arrow[r, twoheadrightarrow] & \frac{H^1(F_{\vec{n},v},M_s)}{H^1_*(F_{\vec{n},v},M_s)} \arrow[d, "p_v"] \\
		0 \arrow[r] & H^1_{*}(F_{\infty,v},M_s)^{\Omega_{\vec{n}}} \arrow[r] &  H^1(F_{\infty,v},M_s)^{\Omega_{\vec{n}} }\arrow[r] & \left(\frac{H^1(F_{\infty,v},M_s)}{H^1_*(F_{\infty,v},M_s)} \right)^{\Omega_{\vec{n}}},
	\end{tikzcd}
	\end{equation}

where $H^1_{*}(F_{\vec{n},v},M_s)=\begin{cases} H^1_{I_v}(F_{\vec{n},v},M_s) \text{ if } v=\gq\in \{\gp,\gp^c\}, \text{ and} \\ H^1_{\mathrm{ur}}(F_{\vec{n},v},M_s) \text{ if not.}  \end{cases} $

\textbf{The case $v=\gq$.}
When $v|p$, both $\ell$ and $g$ are isomorphisms;  $\ell$ is an isomorphism by definition, whereas $g$ is an isomorphism by the inflation-restriction exact sequence and \textbf{(Tors)}.  Thus, the snake lemma implies that $\ker(p_v)=0$. 

\textbf{The case of archimedean $v$.}
Here, we have $\coker(\ell)=0$ \cite[Sect A2.4]{perrinriou00b}, and also $\ker(g)=0$, since $v$ splits completely. 

\textbf{The case $v\nmid p$.} Since $\coker(\ell)=0$, (this follows from the hypothesis \textbf{(Tors)} and using inflation-restriction) we know that $\ker(g)$ surjects onto $\ker(p_v)$ by the snake lemma. Now $\ker(g)=H^1\big(F_{\infty,v}/F_{\vec{n},v}, M_s^{G_{F_\infty,v}}\big)$.
{Since $v \nmid p$, $F_{\infty,v}$ is the unique unramified $\Z_p$-extension of $F_v$. Hence $\Gal(F_{\infty,v}/F_{\vec{n},v})$ is topologically generated by a single element $\gamma_n$. For all but finitely many $s\in \Z$, $M_s^{G_{F_{\vec{n},v}}}$ is finite for every $\vec{n}$; using Lemma \ref{divisible} below, this implies that the order of $\ker(g)$ is bounded by that of $M_s^{G_{F_\infty,v}}/(M_s^{G_{F_\infty,v}})_{\mathrm{div}}$. }

\end{proof}

\begin{lemma}\label{divisible}
Suppose that $N$ is a $\Zp$-module with an action of $\Gamma\cong \Zp$ so that $N^\Gamma$ is finite. 
Then the order of $H^1(\Gamma, N)$ is bounded by that of $N/N_{\mathrm{div}}$.
\end{lemma}

\begin{proof}
Fix a topological generator $\gamma$ of $\Gamma$. Since $\Gamma$ is pro-$p$ cyclic, we know that $H^1(\Gamma,N)\cong N/(\gamma-1)N$. Since $N^\Gamma=\ker(N \xrightarrow{\gamma-1}N )$ is finite, $(\gamma-1)N$ must contain the maximal divisible subgroup $N_{\mathrm{div}}$.
\end{proof}
\begin{proposition}\label{prop:one}
Suppose that the signed Selmer group  $\Sel_{\underline{I}^{{c}}}(M/F_\infty)$ is $\Lambda(\Omega)$-cotorsion. Then the defining sequence 
$$0 \rightarrow \Sel_{\underline{I}}(M^*/F_\infty) \rightarrow H^1(F_\Sigma/F_\infty, M^*) \rightarrow \mathcal{P}_{\Sigma,\underline{I}}(M^*/F_\infty) \rightarrow 0$$
is short exact.
\end{proposition}
\begin{proof}
Suppose $\Sel_{\underline{I}^{{c}}}(M/F_\infty)$ is cotorsion as a $\Lambda(\Omega)$-module. Then for all but finitely many $s \in \Z$, $\big(\Sel_{\underline{I}^{{c}}}(M/F_\infty) \otimes \chi^s\big)^{\Omega_{{n}}} \cong \big(\Sel_{\underline{I}^{{c}}}(M_s/F_\infty)\big)^{\Omega_{{n}}} $ is finite for every ${n}$. Therefore, by Theorem \ref{thm:mainKIM}, for all by finitely many $s$, $\Sel_{\underline{I}^c}(M_s/F_{{n}})$ is finite for every ${n}$. Hence for those $s$ and for all ${n}$, by \cite[Proposition 4.13]{Greenberg},
we obtain that the cokernel of the map 
$$f_{{n},s}:H^1(F_\Sigma/F_{{n}}, M_{-s}^*) \rightarrow \mathcal{P}_{\Sigma,\underline{I}}(M_{-s}^*/F_{{n}}) $$
is the Pontryagin dual of $H^0(F_{{n}},M_s)$ (note that the local components  $H^1_{I_\gq}(F_{n,\gq},M^*_{-s})$ and $H^1_{\mathrm{ur}}(F_{n,v},M^*_{-s})$ for $v \nmid p$ of  $\mathcal{P}_{\Sigma,\underline{I}}(M_{-s}^*/F_{{n}})$ are divisible groups by \cite[Lemma 3.16]{Dion2022} and \cite[Lemma 1.6]{ponsinet}). By hypothesis \textbf{(Tors)}, $H^0(F,M)=0$ and this implies that $H^0(F_\infty,M)=0$. Also $M_s \cong M$ as a $\Gal(\overline{F}/F_\cyc)$-module and hence as a $\Gal(\overline{F}/F_\infty)$-module. Therefore, $H^0(F_\infty,M_s)=0$ and hence $H^0(F_{{n}},M_s)=0$ for any ${n}$. This gives that the maps $f_{{n},s}$ are surjective for all ${n}$ and for all but finitely many $s\in \Z$. Passing to the direct limit of relative to restriction maps, the surjectivity of $f_{{n},s}$ implies the surjectivity of the following map: $$f_{\infty,s}:H^1(F_\Sigma/F_{\infty}, M_{-s}^*) \rightarrow \mathcal{P}_{\Sigma,\underline{I}}(M_{-s}^*/F_{\infty}).$$ Twisting the map $f_{\infty,s}$ by $\chi_{|_\Gamma}^{-s}$ we obtain the map
$$f_{\infty}:H^1(F_\Sigma/F_{\infty}, M^*) \rightarrow \mathcal{P}_{\Sigma,\underline{I}}(M^*/F_{\infty}),$$ which is hence a surjection.
\end{proof}

\begin{corollary}\label{prop:check}
If
the Bloch-Kato Selmer group $\Sel_{\mathrm{BK}}(M^*/F)$ is finite then the Selmer group  $\Sel_{\underline{I}}(M^*/F_\infty)$ is $\Lambda[[\Omega]]$-cotorsion.
\end{corollary}
\begin{proof}
By \cite[Lemma 1.11]{ponsinet}, $\Sel_{\mathrm{BK}}(M^*/F) =\Sel_{\underline{I}}(M^*/F)$. By the control theorem \ref{thm:mainKIM}, $\Sel_{\underline{I}}(M^*/F_\infty)^\Omega$ is finite  and hence by Nakayama's lemma (see \cite{BalisterHowson}) we obtain that $\Sel_{\underline{I}}(M^*/F_\infty)$ is $\Lambda[[\Omega]]$-cotorsion. 
\end{proof}
\section{Preservation of vanishing of $\mu$-invariants under congruences }\label{sec: muinvariance}

The goal of this subsection is to deduce the vanishing of $\mu$-invariant associated to a motive $\mathcal{M}^\prime$ from that of the motive $\mathcal{M}$ in the case that their residual Galois representations are isomorphic. We first remind the reader of the definition of $\mu$-invariants, then state the theorem, before embarking on its proof.

\subsection{Some facts about $\mu$-invariants}\label{sec:mu}
We recall Howson's treatment of $\mu$-invariants in  \cite{Howson_02}. 

Let $G$ be a pro-$p$ $p$-adic Lie group without $p$-torsion. 
\begin{definition}{ \cite[equation 33]{Howson_02}}
    For a finitely generated  $\Z_p[[G]]$-module $N$, its $\mu$-invariant is 
$$\mu(N)=\sum_{i \geq 0}\rank_{\mathbb{F}_p[[G]]}\Big(p^i\big(N(p)\big)/p^{i+1}\Big),$$
where $N(p)$ is the submodule of $N$ consisting of its elements annihilated by some power of $p$.
\end{definition}
Denote by $N[p]$ the $p$-torsion of $N$.
\begin{lemma}\label{lemma:imp} For $N$ as in the above definition, we have

$$\rank_{\mathbb{Z}_p[[G]]}(N)=\rank_{\mathbb{F}_p[[G]]}(N/pN) - \rank_{\mathbb{F}_p[[G]]}(N[p]).$$
Further, $\mu(N)=0$ is equivalent to
$$\rank_{\mathbb{Z}_p[[G]]}(N)=\rank_{\mathbb{F}_p[[G]]}(N/pN).$$

\end{lemma}
\begin{proof}
   The first statement is \cite[Cor. 1.10]{Howson_02}.
   
   For the second statement, first suppose $\mu(N)=0$. Then in particular, $$\rank_{\mathbb{F}_p[[G]]}(N(p)/pN(p))=0.$$ But $\rank_{\mathbb{F}_p[[G]]}(N(p)/pN(p))=0$ is equivalent to $\rank_{\mathbb{F}_p[[G]]}(N[p])=0$ by \cite[equations 42 and 43]{Howson_02}. 
   
   Conversely, suppose $\rank_{\mathbb{F}_p[[G]]}(N[p])=0$. We have $(p^iN)[p]\subseteq N[p],$ so that $\rank_{\mathbb{F}_p[[G]]}((p^iN)[p]))\leq   \rank_{\mathbb{F}_p[[G]]}(N[p])=0.$  
   But we have again by \cite[equations 42 and 43]{Howson_02} that $\rank_{\mathbb{F}_p[[G]]}((p^iN)[p])=0$ is equivalent to $$\rank_{\mathbb{F}_p[[G]]}(p^iN(p)/p^{i+1}N(p))=0.$$ Thus, $\rank_{\mathbb{F}_p[[G]]}(N[p])=0$ implies $\rank_{\mathbb{F}_p[[G]]}(p^iN(p)/p^{i+1}N(p))=0$ for all $i$, which is the same as $\mu(N)=0$.
\end{proof}

\subsection{The Theorem}\label{sec:tt}

Let $\mathcal{M}$ and $\mathcal{M}^\prime$ be motives with $G_F$-stable $\Z_p$-lattices $T$ and $T^\prime$ inside the the $p$-adic realizations satisfying the hypotheses \textbf{(H.-T.), (Cryst.), (Tors.), (Fil.), (Slopes).} 
In this section, we make the following assumption.

\begin{center}

\noindent \textbf{(Cong.)} $T/pT \cong T^\prime/pT^\prime $ as $G_F$-representations.
\end{center}

 Let $\mathcal{X}:=\mathcal{X}_{\underline{I}}(M^*/F_\infty)$ be the Pontryagin dual of the Selmer group $\Sel_{\underline{I}}(M^*/F_\infty)$ and $\mathcal{X}^\prime$ that of $\Sel_{\underline{I}}(M^{\prime*}/F_\infty)$. Recall that $\Omega=\Zp^2$.
 \begin{theorem}\label{them:congruence_mu}
 Suppose that $\mathcal{X}$ and $\mathcal{X}^\prime$ have the same $\Zp[[\Omega]]$-rank $r$. Then $\mu(\mathcal{X})=0$ if and only if $\mu(\mathcal{X^\prime})=0$.
\end{theorem}

Let us first sketch the proof. For the sake of readability, we let $\Sel(M^*)$ denote $\Sel_{\underline{I}}(M^*/F_\infty)$ and $\Sel(M^{\prime*})$ denote $\Sel_{\underline{I}}(M^{\prime*}/F_\infty)$. We introduce auxiliary Selmer groups $\Sel^{\Sigma_0}(M^*[p])=\Sel_{\underline{I}}^{\Sigma_0}(M^*[p]/F_\infty)$ (and similarly $\Sel^{\Sigma_0}(M^{\prime*}[p])$), where the local condition is changed at certain ramified primes ($\Sigma$ stands for `set,' and the subscript $0$ indicates ramification.) 

The main idea is to set up two exact sequences and compare them via an isomorphism $(I)$ of $\Zp[[\Omega]]$-modules:

\begin{equation}\label{ideaofproof}
	\begin{tikzcd}
	0 \arrow[r] & \Sel(M^*)[p] \arrow[r] & \Sel^{\Sigma_0}(M^*[p]) \arrow[d, "(I)"]  \arrow[r] & \text{(cotorsion $\mathbb{F}_p[[\Omega]]$-module)} \\
	0 \arrow[r] & \Sel(M^{\prime*})[p] \arrow[r] & \Sel^{\Sigma_0}(M^{\prime*}[p])  \arrow[r] & \text{(cotorsion $\mathbb{F}_p[[\Omega]]$-module)}
	\end{tikzcd}
\end{equation}

Recall that $\mathcal{X}$ denotes the Pontryagin dual of $\Sel(M^*)$. Similarly, let $\mathcal{Y}$ and $\mathcal{Y^\prime}$ be the Pontryagin duals of $\Sel^{\Sigma_0}(M^*[p])$ and $\Sel^{\Sigma_0}(M^{\prime*}[p])$. 
 Dualizing the first exact sequence gives
\begin{equation}\label{eq:fact2}
\rank_{\mathbb{F}_p[[\Omega]]}\big(\mathcal{Y}\big)= \rank_{\mathbb{F}_p[[\Omega]]}\big(\mathcal{X}/p\mathcal{X}\big).
\end{equation}
Recall we denote by $r$ the $\Zp[[\Omega]]$-rank of $\mathcal{X}$. 
If  $\mu(\mathcal{X})=0$, then by Lemma \ref{lemma:imp},  
$$\rank_{\mathbb{F}_p[[\Omega]]}\big(\mathcal{X}/p\mathcal{X}\big)=r$$ and hence by \eqref{eq:fact2}, 
$\mathcal{Y}$ has $\mathbb{F}_p[[\Omega]]$-rank $r$ as well. 

We now use (I) to reverse our steps: By (I), $\mathcal{Y^\prime}$ has $\mathbb{F}_p[[\Omega]]$-rank $r$ as well. Dualizing the second exact sequence results in an analogue of \eqref{eq:fact2}, which gives us $$\rank_{\mathbb{F}_p[[\Omega]]}\big(\mathcal{X^\prime}/p\mathcal{X^\prime}\big)=r.$$

Since $\mathcal{X^\prime}$ was assumed to have rank $r$, we get 
$$\rank_{\Z_p[[\Omega]]}\big(\mathcal{X^\prime}\big)=\rank_{\mathbb{F}_p[[\Omega]]}\big(\mathcal{X^\prime}/p\mathcal{X^\prime}\big)=r.$$

Again by Lemma \ref{lemma:imp}, $\mu(\mathcal{X^\prime})=0$.


\begin{corollary}\label{torsioncorollary}
 Suppose that $\mathcal{X}$ and $\mathcal{X}^\prime$ are $\Zp[[\Omega]]$-torsion. Then $\mu(\mathcal{X})=0$ if and only if $\mu(\mathcal{X^\prime})=0$.
\end{corollary}

\begin{proof}
    This is the theorem with $r=0$.
\end{proof}

\subsection{The auxiliary Selmer group, the isomorphism (I), and the exact sequences (\ref{ideaofproof}) }
Recall that what we need to do to make the proof work is to define the auxiliary Selmer group, establish the isomorphism (I), and prove the exactness in the diagram \ref{ideaofproof}. To this end, we generalize Ponsinet's construction, cf. \cite[Section 3.3]{ponsinet}.

\subsubsection{The auxiliary Selmer group}
Let $\Sigma_0 \subset \Sigma$ containing all the primes of ramification of $M^*$ but neither the primes of $F$ dividing $p$ nor the archimedean places. Let $\Sigma_0^\prime$ be the primes of $F_\infty$ above $\Sigma_0$. Recall that $\gq$ is the unique prime above $p$ in $F_\infty.$ From  \cite[Lemma 3.5.3]{MazurRubinKolyvaginsystems} and  proofs analogous to \cite[Lemma 3.5]{ponsinet}, we know that 
     $H^1(F_{\infty,\gq}, M^*[p]) \cong H^1(F_{\infty,\gq}, M^*)[p],$
    so that  
$$H^1_{I_\gq}(F_{\infty,\gq},M^*)[p]\subset H^1(F_{\infty,\gq},M^*[p]).$$
We set, just as Ponsinet did,
$$H^1_{I_\gq}(F_{\infty,\gq},M^*[p]):=H^1_{I_\gq}(F_{\infty,\gq},M^*)[p].$$
Thus, the following local condition is well-defined: 
$$\mathcal{P}_{\Sigma  \backslash \Sigma_0,\underline{I}}(M^*[p]/F_{\infty}):=\prod_{w \in \Sigma^\prime \backslash \Sigma_0^\prime, w \nmid p}H^1(F_{w,\mathrm{ur}}, M^*[p])\times \prod_{\gq \mid p}\frac{\bigoplus_{w \mid \gq}H^1(F_{\infty,w}, M^*[p])}{H^1_{I_\gq}(F_{\infty,\gq},M^*[p])}.$$
\begin{definition}
We define the \textit{$\Sigma_0$-non-primitive $\underline{I}$-Selmer group of $M^*[p]$ over $F_\infty$} as
$$\Sel_{\underline{I}}^{\Sigma_0}(M^*[p]/F_\infty):=\ker\Big(H^1(F_\Sigma/F_{\infty}, M^*[p]) \rightarrow \mathcal{P}_{\Sigma \backslash \Sigma_0,\underline{I}}(M^*[p]/F_{\infty})\Big).$$
We also use the abbreviated notation $\Sel^{\Sigma_0}(M^*[p])$.
    
\end{definition}

\subsubsection{The isomorphism (I)}\label{(I)}
We would like to prove that as $\mathbb{F}_p[[\Omega]]$-modules, $$\Sel^{\Sigma_0}(M^*[p]) \cong \Sel^{\Sigma_0}({M^\prime}^*[p]).  $$ under the hypothesis \textbf{(Cong.)}. But  \textbf{(Cong.)} implies that  $M^*[p] \cong {M^\prime}^*[p]$ as $G_F$-representations. Hence 
\begin{align*}  H^1(F_\Sigma/F_\infty, M^*[p])&\cong H^1(F_\Sigma/F_\infty, {M^\prime}^*[p]) \text{ and} \\
H^1(F_{w,\mathrm{ur}}, M^*[p])&\cong H^1(F_{w,\mathrm{ur}}, {M^\prime}^*[p]).
\end{align*}
Thus, all that remains is proving that $$H^1_{I_\gq}(F_{\infty,\gq},M^*[p])\cong H^1_{I_\gq}(F_{\infty,\gq},{M^\prime}^*[p]).$$
We prove that their duals are isomorphic: The Tate-pairing \eqref{TatePairing} gives that the Pontryagin dual of $H^1_{I_\gq}(F_{\infty,\gq},M^*[p])$ is $\im(\col_{T,I_\gq}^\infty)/p. $
Thus, our assertion follows from the two lemmas:
\begin{lemma}
Assume \textbf{(Cong.)} holds. Then as $\Z_p[[\Omega_p]]$-modules, $H^1_{\mathrm{Iw}}(F_{\infty,\gq},T)/p\cong H^1_{\mathrm{Iw}}(F_{\infty,\gq},T^\prime)/p.$
\end{lemma}
\begin{lemma}
Assume \textbf{(Cong.)} holds. If $z\in H^1_{\mathrm{Iw}}(F_{\infty,\gq},T)$ and $z^\prime \in H^1_{\mathrm{Iw}}(F_{\infty,\gq},T^\prime)$ have the same image under the isomorphism given in the above lemma, then the Coleman maps $\col_{T,i}^\infty(z)$ and $\col_{T^\prime,i}^\infty(z^\prime)$ are congruent modulo $p$.
\end{lemma}
See subsection \ref{congruencelemmas} for their proofs.

\subsubsection{The exact sequences (\ref{ideaofproof}})
We prove exactness of the top sequence, as the second sequence is exact for the same reasons.
Recall that we let $\Sigma_0 \subset \Sigma$ contain all the primes of ramification of $M^*$ but neither the primes of $F$ dividing $p$ nor the archimedean places. Let $\Sigma_0^\prime$ be the primes of $F_\infty$ above $\Sigma_0$.

 We would like to compare the auxiliary Selmer group $\Sel_{\underline{I}}^{\Sigma_0}(M^*[p]/F_\infty)$ to the ($p$-torsion of the) $\Sigma_0$-non-primitive $\underline{I}$-Selmer group 
 , which we define by 

$$\Sel_{\underline{I}}^{\Sigma_0}(M^*/F_\infty)=\ker\Big(H^1(F_\Sigma/F_{\infty}, M^*) \rightarrow \mathcal{P}_{\Sigma \backslash \Sigma_0,\underline{I}}(M^*/F_{\infty})\Big).$$

Comparing their local conditions gives a map $$\theta: \Sel_{\underline{I}}(M^*/F_\infty) \xrightarrow{\theta}  \Sel_{\underline{I}}^{\Sigma_0}(M^*/F_\infty).$$
The map appearing in the exact sequence (\ref{ideaofproof}) is the $p$-torsion part of the map $\theta$ (which we will denote by $\theta_p$, composed with the isomorphism of the following lemma: 

\begin{lemma}

We have the isomorphism
\begin{equation}\label{eq:Selmcong}
\Sel_{\underline{I}}^{\Sigma_0}(M^*/F_\infty)[p]\cong 
\Sel_{\underline{I}}^{\Sigma_0}(M^*[p]/F_\infty) .
\end{equation}
\end{lemma}
\begin{proof}
We match the global term and the local conditions of the Selmer groups in question, i.e. this isomorphism reduces to
\begin{enumerate}
    \item $H^1(F_\Sigma/F_\infty, M^*[p]) \cong H^1(F_\Sigma/F_\infty, M^*)[p],$
    \item $H^1(F_{\infty,\gq}, M^*[p]) \cong H^1(F_{\infty,\gq}, M^*)[p]$,
    \item $H^1(F_{w,\mathrm{ur}}, M^*[p]) \cong H^1(F_{w,\mathrm{ur}}, M^*)[p]$ for any non-archimedean prime $w \in \Sigma_0^\prime$ not dividing $p$ or a prime of ramification of $M^*$, where $F_{w,\mathrm{ur}}$ is the maximal unramified extension of $F_w$.
\end{enumerate}

But these isomorphisms are induced by the exact sequence $$0 \rightarrow M^*[p] \rightarrow M^* \xrightarrow{p} M^* \rightarrow 0$$ of $G_{F_\infty}$-modules, analogously to  \cite[Lemma 3.5]{ponsinet}, see also  \cite[Lemma 3.5.3]{MazurRubinKolyvaginsystems}.

\end{proof}
\begin{lemma}
    As a  $\mathbb{Z}_p[[\Omega]]$-module, the Pontryagin dual of $\coker(\theta)$ is torsion. 
\end{lemma}
\begin{proof}
By Proposition \ref{prop:one}, we have the following fundamental diagram.

\begin{equation}\label{fundamental1}
	\begin{tikzcd}
	0 \arrow[r] & \Sel_{\underline{I}}(M^*/F_\infty) \arrow[d, "\theta"] \arrow[r] &  H^1(F_\Sigma/F_\infty,M^*)\arrow[d, "="]  \arrow[r, "\delta"] & \mathcal{P}_{\Sigma,\underline{I}}(M^*/F_\infty) \arrow[d, "h_{\Sigma_0}"] \\
		0 \arrow[r] & \Sel_{\underline{I}}^{\Sigma_0}(M^*/F_\infty)\arrow[r] &  H^1(F_\Sigma/F_\infty,M^*)\arrow[r] & \mathcal{P}_{\Sigma \backslash \Sigma_0,\underline{I}}(M^*/F_{\infty}).
	\end{tikzcd}
	\end{equation}
By the snake lemma we have an injection $\Sel_{\underline{I}}(M^*/F_\infty) \hookrightarrow \Sel_{\underline{I}}^{\Sigma_0}(M^*/F_\infty)$ 
with cokernel isomorphic to $\ker(h_{\Sigma_0}) \cap \im(\delta)$, where $h_{\Sigma_0}$ is the projection map on local conditions. (Note that  $\mathcal{P}_{\Sigma \backslash \Sigma_0,\underline{I}}(M^*/F_{\infty})$ has the same local components as $\mathcal{P}_{\Sigma,\underline{I}}(M^*/F_{\infty})$ but removing the places in $\Sigma_0$.) Now,

$$\ker(h_{\Sigma_0})=\prod_{w \in \Sigma_0^\prime}\frac{H^1(F_{\infty,w}, M^*)}{H^1_{\mathrm{ur}}(F_{\infty,w}, M^*)}.$$
Since $w \in \Sigma_0^\prime$, $w$ is not above $p$ and hence $F_{\infty,w}$ is the unique unramified $\Z_p$-extension of $F_v$ where $v$ is a prime of $F$ below $w$ (see \cite[Section 5.2]{LeiSujatha}). In this case $H^1(F_{\infty,w}, M^*)$ is a cotorsion $\Z_p[[\Gamma]]$-module (see \cite[Proposition 2]{greenberg89}) and $H^1_{\mathrm{ur}}(F_{\infty,w}, M^*)$ is trivial (see \cite[Section A.2.4]{PR95}).
Hence $\ker(h_{\Sigma_0})$ is a cotorsion $\Zp[[\Omega]]$-module and therefore it is so for  $\ker(h_{\Sigma_0}) \cap \im(\delta)=\coker(\theta)$.  Hence the conclusion.

\end{proof}
Multiply the nonzero terms in the following exact sequence by $p$:
$$0 \rightarrow \Sel_{\underline{I}}(M^*/F_\infty) \xrightarrow{\theta}  \Sel_{\underline{I}}^{\Sigma_0}(M^*/F_\infty) \rightarrow \coker(\theta) \rightarrow 0, $$
Applying the snake lemma to this multiplication by $p$, we get the exact sequence 
$$0 \rightarrow \Sel_{\underline{I}}(M^*/F_\infty)[p] \xrightarrow{\theta_p} \Sel_{\underline{I}}^{\Sigma_0}(M^*/F_\infty)[p]\rightarrow \coker(\theta)[p] \rightarrow \Sel_{\underline{I}}(M^*/F_\infty)/p.$$
We can quotient out the first nonzero term and its image under $\theta_p$ to see that $\coker(\theta_p)$ injects into $\coker(\theta)[p]$. Thus, the Pontryagin dual of $\coker(\theta_p)$ is a quotient of the dual of $\coker(\theta)[p]$, which is $\mathbb{F}_p[[\Omega]]$-torsion by the previous lemma. Hence the Pontryagin dual of $\coker(\theta_p)$ is also $\mathbb{F}_p[[\Omega]]$-torsion. 
Thus, combining $\theta_p$ with the isomorphism  \eqref{eq:Selmcong}, we obtain  an injection 
$$\Sel_{\underline{I}}(M^*/F_\infty)[p] \hookrightarrow \Sel_{\underline{I}}^{\Sigma_0}(M^*[p]/F_\infty)$$ with cokernel a cotorsion $\mathbb{F}_p[[\Omega]]$-module; this is exactly what was required.

\subsubsection{Congruence lemmas in Iwasawa algebras}\label{congruencelemmas}
Recall that $\gq$ is the unique prime above $p$ in $F_\infty.$ We complete the outstanding proof of the two last lemmas in subsection \ref{(I)}, which may be of independent interest.
\begin{lemma}\label{lem: modp}
Assume \textbf{(Cong.)} holds. Then as $\Z_p[[\Omega_p]]$-modules, $H^1_{\mathrm{Iw}}(F_{\infty,\gq},T)/p\cong H^1_{\mathrm{Iw}}(F_{\infty,\gq},T^\prime)/p.$
\end{lemma}
\begin{proof} Let $\mathbf{A}_{\Q_p}^+$ be the ring $\Z_p[[\pi]]$ equipped with a semilinear action of the Frobenius $ \varphi$ acting as the absolute Frobenius on $\Z_p$ and on the uniformizer $\pi$  as $\varphi(\pi)=(\pi+1)^p-1$ with a Galois action given by $g(\pi)=(\pi+1)^{\chi(g)}-1$ where $g \in \Gal(\Q_p(\mu_{p^\infty})/\Q_p)$. Then the (cyclotomic) Wach module $N_{\Q_p}(T)$ is equipped with a filtration 
$$\mathrm{Fil}^i(N_{\Q_p}(T))=\{x \in N_{\Q_p}(T), \varphi(x) \in (\varphi(\pi)/\pi)^i N_{\Q_p}(T)\}.$$
Since the Hodge-Tate weights of $T$ and $T^\prime$ are in $[0,1]$, under \textbf{(Cong.)}, from \cite[Théorème IV.1.1]{berger04}, we deduce that
\begin{equation}\label{eq:Wach}
N_{\Q_p}(T)/pN_{\Q_p}(T) \cong N_{\Q_p}(T^\prime)/pN_{\Q_p}(T^\prime)
\end{equation}

as  $\mathbf{A}_{\Q_p}^+$-modules and the isomorphism is compatible with the filtration, the Galois action and the action of the Frobenius operator $\varphi$. Since
$H^1_{\mathrm{Iw}}(F_{\gq}(\mu_{p^\infty}),T) \cong N_{\Q_p}(T)^{\psi=1}$ as
Galois modules, we have the congruence of the first Iwasawa cohomology groups
$$H^1_{\mathrm{Iw}}(F_{\gq}(\mu_{p^\infty}),T)/p \cong H^1_{\mathrm{Iw}}(F_{\gq}(\mu_{p^\infty}),T^\prime)/p. $$
The proof follows by tensoring with the Yager module (cf. Section \ref{sec Yager}) and taking the $\Delta$-invariance of
$$H^1_{\mathrm{Iw}}(L_\infty(\mu_{p^\infty}),T)\cong N_{L_\infty}(T)^{\psi=1}\cong N_{\Q_p}(T)^{\psi=1} \widehat{\otimes} S_{L_\infty/\Q_p} \cong H^1_{\mathrm{Iw}}(F_\gq(\mu_{p^\infty}),T) \widehat{\otimes} S_{L_\infty/\Q_p}.$$
\end{proof}

In the last lemma, we  use the identification $\mathds{D}_{\mathrm{cris},\gq}(T)=N_{\Q_p}(T)/pN_{\Q_p}(T)$ and fix Hodge-compatible bases (see Section \ref{sub:HC}) of $\mathds{D}_{\mathrm{cris},\gq}(T)$ and $\mathds{D}_{\mathrm{cris},\gq}(T^\prime)$ such that their images are the same under the isomorphism 
\eqref{eq:Wach}.

\begin{lemma}\label{lem:conImage}
Assume \textbf{(Cong.)} holds. If $z\in H^1_{\mathrm{Iw}}(F_{\infty,\gq},T)$ and $z^\prime \in H^1_{\mathrm{Iw}}(F_{\infty,\gq},T^\prime)$ have the same image under the isomorphism given in Lemma \ref{lem: modp}, then the Coleman maps $\col_{T,i}^\infty(z)$ and $\col_{T^\prime,i}^\infty(z^\prime)$ are congruent modulo $p$.
\end{lemma}
\begin{proof}
Recall that $L_m$ is the unramified $\Z_p$-extension of degree $p^m$ over $\Q_p$ and $L_{m,cyc}$ is the cyclotomic $\Z_p$-extension of $L_m$. We identify $H^1_{\mathrm{Iw}}(F_{\infty,\gq},T)$ with $H^1_{\mathrm{Iw}}(k_{\infty},T)\cong\varprojlim_{L_m}H^1_{\mathrm{Iw}}(L_{m,cyc},T)$ and note that we have the following commutative diagram.

\[
\begin{tikzcd}
  H^1_{\mathrm{Iw}}(k_{\infty},T)/p \arrow[r, "~"] \arrow[d, " \cong"]
    & H^1_{\mathrm{Iw}}(L_{m,cyc},T)/p \arrow[d, "\cong"] \\
  H^1_{\mathrm{Iw}}(k_{\infty},T^\prime)/p \arrow[r]
&  H^1_{\mathrm{Iw}}(L_{m,cyc},T^\prime)/p \end{tikzcd}
\]
The horizontal arrows are the natural projection maps. The left vertical arrow is the isomorphism from Lemma \ref{lem: modp}. The right vertical arrow is an isomorphism which follows from \cite[Eq. (18) in Section 3.2]{ponsinet}.  Let $z = \varprojlim_{L_m} z_m$ where $z_m \in H^1_{\mathrm{Iw}}(L_{m,cyc},T)$ and correspondingly define $z_m^\prime$. Our hypotheses imply that $z_m$ and $z_m^\prime$ have the same image under the isomorphism given by the right vertical arrow of the commutative square given above.

From  \eqref{two-variable-coleman}, 
recall that
\begin{equation}\label{eq:label1}
\col_{T,i}^\infty((z_m))=(y_{L_\infty/\Q_p}\otimes 1) \circ (\varprojlim_{L_m} \col_{T,L_m,i}(z_m)).
\end{equation}

By \cite[Proposition 3.4]{ponsinet}, $\col_{T,L_m,i}(z_m)$ and $\col_{T^\prime,L_m,i}(z_m^\prime)$ are congruent modulo $p$.
The result follows by taking inverse limits and tensoring with the Yager module using \eqref{eq:label1}.
\end{proof}

\begin{remark}
   Note that if we had assumed that $\mathcal{X}_{\underline{I}^c}(M/F_\infty)$ is a torsion $\Lambda(\Omega)$-module, then the map $\delta$ will be a surjection (see Proposition \ref{prop:one}). So it poses a natural question of when can we relate the torsionness of the module $\mathcal{X}_{\underline{I}^c}(M/F_\infty)$ with that of $\mathcal{X}_{\underline{I}}(M^*/F_\infty)$. The result is the following. 
   
   If the abelian variety is polarized then $\mathcal{X}_{\underline{I}^c}(M/F_\infty)$ is a torsion $\Lambda(\Omega)$-module if and only if $\mathcal{X}_{\underline{I}}(M^*/F_\infty)$ is a torsion $\Lambda(\Omega)$-module \cite[Corollary 5.13]{Dion2022}.
\end{remark}
\bibliographystyle{alpha}
\bibliography{main}
\end{document}

%% file: canonicalheader.tex
\usepackage{amssymb}
\usepackage{amscd,amsfonts,amsmath,amssymb, bbm,dsfont, comment}
\usepackage{amsmath}
\usepackage{amsfonts}
\usepackage{amssymb}\usepackage{enumerate,epsf,fancyhdr,float,graphicx,tabularx}
\usepackage{latexsym,mathrsfs,multirow}
\usepackage{wasysym}
\usepackage{xypic}
\usepackage[all]{xy}\usepackage[OT2,T1]{fontenc}
\usepackage[modulo,mathlines,displaymath, running]{lineno}
\usepackage{amsmath,mathrsfs,enumerate,wasysym}
\usepackage{xypic,hhline}
\usepackage[all]{xy}
\usepackage[OT2,T1]{fontenc}
\usepackage[modulo,mathlines,displaymath]{lineno}
\usepackage{tikz,tikz-cd}
\usepackage{dashrule}
\usetikzlibrary{matrix}
\usepackage[modulo,mathlines,displaymath]{lineno}

\newtheorem{theorem}{Theorem}[section]
\newtheorem{conj}[theorem]{Conjecture}
\DeclareSymbolFont{cyrletters}{OT2}{wncyr}{m}{n}\DeclareMathSymbol{\Sha}{\mathalpha}{cyrletters}{"58}

\renewcommand{\phi}{{\varphi}}
\newcommand{\Qp}{\mathbf{Q}_p}
\newcommand{\Zp}{\mathbf{Z}_p}

\renewcommand{\geq}{\geqslant}
\renewcommand{\leq}{\leqslant}

\newcommand{\cyc}{\mathrm{cyc}}
\newcommand{\Dcrisv}{\mathbb{D}_{\mathrm{cris},v}}
\newcommand{\Dcrisq}{\mathbb{D}_{\mathrm{cris},\gq}}
\newcommand{\HIw}{H^1_{\mathrm{Iw}}}
\newcommand{\Fil}{\mathrm{Fil}}
\newcommand{\cris}{\mathrm{cris}}

\newcommand{\links}{\left(\begin{array}{cc}}
\newcommand{\rechts}{\end{array}\right)}
\newcommand{\bai}{\left[\begin{array}{cc}}
\newcommand{\dai}{\end{array}\right]}
\newcommand{\hidari}{\left(\begin{array}{c}}
\newcommand{\migi}{\end{array}\right)}

\newcommand{\Q}{\mathbb{Q}}

\newcommand{\Z}{\mathbb{Z}}

\newcommand{\gp}{{\mathfrak p}}
\newcommand{\gq}{{\mathfrak q}}

\newcommand{\Gal}{\operatorname{Gal}}

\newcommand{\Hom}{\operatorname{Hom}}

\newcommand{\rank}{\operatorname{rank}}
\newcommand{\coker}{\operatorname{coker}}

\newcommand{\col}{\operatorname{Col}}

\newcommand{\im}{\operatorname{Im}}
\newcommand{\Sel}{\operatorname{Sel}}

\newcommand{\tensor}{\otimes} 

\renewcommand{\contentsname}{Contents\\{\footnotesize\normalfont(A table
of contents should normally not be included)}}

\newtheorem{auxiliary proposition}[theorem]{Auxiliary Proposition}

\newtheorem{corollary}[theorem]{Corollary}

\newtheorem{definition}[theorem]{Definition}

\newtheorem{lemma}[theorem]{Lemma}
\newtheorem{main conjecture}[theorem]{Main Conjecture}
\newtheorem{main theorem}[theorem]{Main Theorem}
\newtheorem{modesty proposition}[theorem]{Modesty Proposition}

\newtheorem{open problem}[theorem]{Open Problem}

\newtheorem{proposition}[theorem]{Proposition}

\newtheorem{remark}[theorem]{Remark}

\newtheorem{convergence lemma}[theorem]{Convergence Lemma}
\newtheorem{corrected lemma}[theorem]{Corrected Lemma}
\newtheorem{growth lemma}[theorem]{Growth Lemma}
\newtheorem{coefficient lemma}[theorem]{Integrality Lemma}
\newtheorem{interpolation lemma}[theorem]{Interpolation Lemma}
\newtheorem{kernel lemma}[theorem]{Kernel Lemma}
\newtheorem{limit lemma}[theorem]{Limit Lemma}
\newtheorem{tandem lemma}[theorem]{Modesty Lemma}
\newtheorem{zero-finding lemma}[theorem]{Zero-Finding Lemma}